\documentclass[a4paper]{amsart}

\usepackage[lmargin=1in,rmargin=1in,tmargin=1in,bmargin=1in]{geometry}
\RequirePackage{amsmath} 
\RequirePackage{amssymb}
\usepackage{amscd,latexsym,amsthm,amsfonts,amssymb,amsmath,amsxtra}
\usepackage[colorlinks=true,urlcolor=blue,citecolor=blue]{hyperref}
\usepackage{color}
\usepackage[all]{xy}
\usepackage[OT2,T1]{fontenc}
\usepackage{bm}
\usepackage{mathtools}
\usepackage{mathrsfs }
\usepackage{xcolor}
\usepackage{comment}
\usepackage{marginnote}
\usepackage{enumitem}
\usepackage{thmtools, thm-restate}

\DeclareSymbolFont{cyrletters}{OT2}{wncyr}{m}{n}
\DeclareMathSymbol{\Sha}{\mathalpha}{cyrletters}{"58}

\let\Re\undefined

\DeclareMathOperator{\Re}{Re}

\newcommand{\bZ}{\mathbb{Z}}

\newcommand{\bQ}{\mathbb{Q}}

\newcommand{\cA}{\mathcal{A}}
\newcommand{\bR}{\mathbb{R}}

\newcommand{\cH}{\mathcal{H}}

\newcommand{\cD}{\mathcal{D}}

\newcommand{\p}{\partial}

\newcommand{\cS}{\mathcal{S}}
\newcommand{\cB}{\mathcal{B}}

\newcommand{\cL}{\mathcal{L}}

\newcommand{\cG}{\mathcal{G}}

\newcommand{\cM}{\mathcal{M}}

\newcommand{\fS}{\mathfrak{S}}

\newcommand{\psum}{\sideset{}{^*}\sum}
\newcommand{\dsum}{\sideset{}{^d}\sum}

\newcommand{\ssum}{\sideset{}{^{\sharp}}\sum}

\def\Re{\operatorname{Re}}

	\newcommand{\Mod}[1]{\ (\mathrm{mod}\ #1)}
        \newcommand{\sym}{\operatorname{sym}}

	\newcommand{\Res}{\operatorname{Res}}

	\newcommand{\RNum}[1]{\uppercase\expandafter{\romannumeral #1\relax}}

\begin{document}
\theoremstyle{plain}
	\newtheorem{thm}{Theorem}[section]
	\newtheorem{cor}[thm]{Corollary}
	\newtheorem{thmy}{Theorem}
        \newtheorem{cory}{Corollary}
	\renewcommand{\thethmy}{\Alph{thmy}}
	\newenvironment{thmx}{\stepcounter{thm}\begin{thmy}}{\end{thmy}}
        \newenvironment{corx}{\stepcounter{cor}\begin{cory}}{\end{cory}}
        
\newtheorem{prob}[thm]{Problem}

\newtheorem{problemalt}{Problem}[thm]        
        
\newenvironment{probp}[1]{
  \renewcommand\theproblemalt{#1}
  \problemalt
}{\endproblemalt}

	\renewcommand{\thecory}{\Alph{cory}}
	\newtheorem{hy}[thm]{Hypothesis}
	\newtheorem*{thma}{Theorem A}
	\newtheorem*{corb}{Corollary B}
	\newtheorem*{thmc}{Theorem C}
        \newtheorem*{thmd}{Theorem D}
	\newtheorem{lemma}[thm]{Lemma}  
	\newtheorem{prop}[thm]{Proposition}
	\newtheorem{conj}[thm]{Conjecture}
	\newtheorem{fact}[thm]{Fact}
	\newtheorem{claim}[thm]{Claim}
	\theoremstyle{definition}
	\newtheorem{defn}[thm]{Definition}
	\newtheorem{example}[thm]{Example}
	\theoremstyle{remark}
	
	\newtheorem{remark}[thm]{Remark}	
	\numberwithin{equation}{section}

\title[The first moment]{The first moment}%
\author{}

\title[]{First Moment of derivatives of $L$-functions in a nonlinear family}%
\author{Tinghao Huang and Zhining Wei}

\address{Mathematics Tower, 231 18th Ave, Columbus, OH 43210 USA}
\email{huang.4939@osu.edu}

\address{School of Mathematics and Statistics, Xi'an Jiaotong University
Xi'an 710049, Shaanxi, People's Republic of China}
\email{wei.863@xjtu.edu.cn}

\maketitle

\begin{abstract}
    In this paper, we prove an asymptotic formula for the moment of the first order derivative of modular $L$-functions at the center of the critical strip, weighted by generalized divisor functions formed by primitive quadratic characters. Such moments were previously studied by Munshi \cite{compMunshiMR2771124,transactionMunshiMR2792990}, which naturally arise in the study of elliptic fibration. To the best of our knowledge, such asymptotic formulae have only been proven in the setting of higher-order derivatives, or under the specialization to dihedral forms.
\end{abstract}


\section{Introduction}
Moments of $L$-functions are among the central objects of study in analytic number theory. Broadly speaking, they concern the asymptotic behavior of sums of $L$-functions over suitable families. Moments encode fundamental information about the distribution of values of $L$-functions and have far-reaching applications to a variety of important problems in arithmetic geometry (see, for example, \cite{BSD-IMR146143, BSD-IIMR179168}) and the theory of automorphic representations (see, for example, \cite{Gan-Gross-PrasadMR3202556}).

The moments of $L$-functions have been extensively studied over many natural families. Among these, the family of quadratic twists of $L$-functions associated with holomorphic cusp forms is of particular interest. Besides its intrinsic importance, this family has significant applications to the study of quadratic twists of elliptic curves and to the Fourier coefficients of half-integral weight modular forms through the Waldspurger--Kohnen--Zagier formula (\cite{WaldspurgerMR646366, Kohnen-ZagierMR629468}).

The first moment of this family has been studied extensively. Let $f$ be a holomorphic Hecke eigenform. The first moment of the
central values of the quadratic twist $L$-functions associated with $f$ is given by 
\[
\sum_{d} L\!\left(\frac12,f\times\chi_d\right)
F\!\left(\frac{d}{X}\right),
\]
where $X>0$ and $F\in C_c^\infty(\mathbb{R})$. One may also consider the
corresponding first moment of the central derivatives,
\[
\sum_{d} L'\!\left(\frac12,f\times\chi_d\right)
F\!\left(\frac{d}{X}\right).
\]
In both sums, $d$ ranges over appropriate subsets of the fundamental
discriminants.

Through the work of many authors, the asymptotic behavior of these first moments
as $X\to\infty$ is now well understood; see, for example,
\cite{GoldfeldHoffsteinPatterson1982,BumpFriedbergHoffstein1990,
MurtyMurty1991,IwaniecMR1081731}. In both cases, one obtains a non-vanishing
main term of order $X$, together with an error term of size
$O(X^\delta)$ for some $0<\delta<1$.
For the first moment of the central values, the error term has been
successively improved; see, for example,
\cite{IwaniecMR1081731,LuoRamakrishnan1997,Munshi2009, PetrowMR3180602,Shen-MR4528422}.

Although precise conjectures for higher moments of quadratic twist
$L$-functions have been formulated (see
\cite{ConreyFarmerKeatingRubinsteinSnaith2005}), proving such asymptotic
formulae becomes substantially more difficult as the order of the moment
increases. A long-standing problem has been to establish asymptotic formulae
for the second moments
\[
\sum_d L\!\left(\frac12,f\times\chi_d\right)^2
F\!\left(\frac{d}{X}\right)
\]
and
\[
\sum_d L'\!\left(\frac12,f\times\chi_d\right)^2
F\!\left(\frac{d}{X}\right),
\]
where in each case $d$ ranges over a suitable subset of the fundamental
discriminants.

Asymptotic formulae for the second moments of both the central values and the
central derivatives were first established in
\cite{Young-SoundMR2677611} and \cite{PetrowMR3180602},
respectively, under the assumption of the Generalized Riemann Hypothesis (GRH)
for the family
\[
\left\{
L(s,f\times\chi_{d_0})
: d_0 \text{ is a fundamental discriminant}
\right\}.
\]
In the same paper, Petrow \cite{PetrowMR3180602} also proved, under GRH, an
asymptotic formula for the mixed first moment
of $L'\!\left(\frac12,f\times\chi_d\right)
L'\!\left(\frac12,g\times\chi_d\right)$,
where $g$ is a holomorphic Hecke eigenform distinct from $f$.

A major breakthrough was achieved by Li \cite{LiMR4768632}, who
unconditionally established an asymptotic formula for the second moment of
$L(\frac12,f\times\chi_d)$. His proof relies on essentially optimal estimates
for bilinear forms involving quadratic characters and Hecke eigenvalues.
Building on Li's ideas, Kumar \cite{KumarMR4765788} established the
corresponding unconditional asymptotic formula for the second moment of
$L'(\frac12,f\times\chi_d)$. More recently,
\cite{jiang2026second} refined this result by different methods.
Li's techniques have also led to further developments. Zhou
\cite{zhou2025momentderivativesquadratictwists} established an unconditional
asymptotic formula for the mixed first moment of
$L'\!\left(\frac12,f\times\chi_d\right)
L'\!\left(\frac12,g\times\chi_d\right)$,
and Huang \cite{Huang2025} subsequently obtained an unconditional asymptotic
formula for $L\!\left(\frac12,f\times\chi_d\right)
L'\!\left(\frac12,g\times\chi_d\right)$.

In this paper, we consider a family of quadratic twists of modular $L$-functions arising from quadratic fibration of elliptic curves over $\mathbb{Q}$.
 To state our result, we let $D<0$ be a fundamental discriminant satisfying $4|D$ and $\chi_D(\cdot) = \big(\frac{D}{\cdot}\big)$ the associated primitive quadratic character. For $m\in\bZ$, we define the weight function:
\[\psi(m,D):=\sum_{r|m}\chi_D(r).\] 
We further let $q$ be a squarefree integer satisfying $(q,6D)=1$ and $f$ a Hecke eigenform of weight $2$ and level $q.$ In addition, we assume that the global root number of $f$, denoted by $\omega_f$, is equal to $-1$. Finally, we denote by $\mathcal{D}$ the set of all integers
\[\mathcal{D}=\{d>0: \mbox{$d$ is a fundamental discriminant, $d\equiv1\Mod{4},$ $(d,qD)=1$ and $\chi_d(q)=1$}\}.\]
We consider the following first moment of the quadratic twists of central derivatives, weighted by $\psi(\cdot,D)$:
$$
\sum_{d \in \mathcal{D}}\psi(d,D)L'(1/2, f\times \chi_d)J\bigg(\frac{d}{X}\bigg).
$$
For $d \in \mathcal{D}$, we have
$$
\psi(d,D)=\sum_{r\mid d}\chi_D(r)
=\prod_{p\mid d}\bigl(1+\chi_D(p)\bigr).
$$
Hence, $\psi(d,D)$ is nonzero if and only if every prime dividing $d$ splits in the imaginary quadratic field $\mathbb{Q}(\sqrt{D})$. Consequently, the above summation is essentially supported on a sparse subset of fundamental discriminants.

Such quadratic twists arise naturally in the study of quadratic fibrations of elliptic curves over $\mathbb{Q}$. The above first moment was investigated by Munshi in \cite{compMunshiMR2771124,transactionMunshiMR2792990}. More precisely, in \cite{compMunshiMR2771124}, Munshi studied the first moment of the $l$-th derivatives of the completed $L$-functions in the above family (although his choice of $\mathcal{D}$ differs slightly from ours) for $D=-4$, and established asymptotic formulae as $X\to\infty$ for $l\ge 8$. In \cite{transactionMunshiMR2792990}, Munshi established the first moment of the first derivative of quadratic twisted modular $L$-functions, when the underlying modular form corresponds to a CM elliptic curve. We also remark that, in \cite{transactionMunshiMR2792990}, Munshi considered more complicated weight functions. 

In this paper we focus on the first order derivative and prove the following:
\begin{thm}\label{thm.main theorem}
Let $X>1$ and $J(x)$ a smooth function compactly supported in $(0,\infty).$ Assume the notations above. Then:
\[\sum_{d\in\cD}\psi(d,D)L^{\prime}(1/2,f\times\chi_d)J\left(\frac{d}{X}\right)=c_{q,D}\cG(q,D)\widetilde{J}(1)L(1,\chi_D)X\log X+O\left(X(\log X)^{\frac{1}{2}}(\log\log X)^3\right),
\]
	where $\widetilde{J}(s)$ is the Mellin transformation of $J(x),$ and $c_{q,D}$ (resp. $\cG(q,D)$) is defined in \eqref{eq.constant in euler} (resp. \eqref{eq. extra term}).
    In particular, $c_{q,D}\mathcal{G}(q,D) \neq 0$.
\end{thm}
We conclude by remarking that, when $D=-4$, our main theorem may be viewed as an improvement of Theorem~1.1 in \cite{compMunshiMR2771124}. 

\subsection{Idea of Proof} The proof of Theorem \ref{thm.main theorem} combines the methods used in \cite{compMunshiMR2771124} and that in \cite{LiMR4768632} and \cite{zhou2025momentderivativesquadratictwists}.  A standard approximate functional equation argument and the application of the Mellin inverse transformation will imply
\[\sum_{d\in\cD}\psi(d,D)L^{\prime}(1/2,f\times\chi_d)J\left(\frac{d}{X}\right)\approx\frac{1}{2\pi i}\int_{(2)}\widetilde{J}(s)X^s\sum_{n\ll X^{1+\varepsilon}}\frac{\lambda_f(n)}{n^{1/2}}W\left(\frac{n}{d\sqrt{q}}\right)\sum_{d\geq1}\frac{\psi(d;D)\chi_d(n)}{d^s}\,ds,\]
where $W$ is a smooth function of rapid decay. In view of quadratic reciprocity, the Dirichlet series in $d$ is roughly $L(s,\chi_{Dn})L(s,\chi_n)$. We then shift the integration line to $\Re(s)=\frac{1}{2}$ and the poles will occur when $n$ or $Dn$ is a perfect square. After picking up the poles, the error term is of the form
\[E\approx\int_{(1/2)}\widetilde{J}(s)X^s\sum_{n\ll X^{1+\varepsilon}}\frac{\lambda_f(n)}{n^{1/2}}W\left(\frac{n}{d\sqrt{q}}\right)L(s,\chi_{Dn})L(s,\chi_n)\,ds.\]
At this stage, the standard method of applying the functional equations for $L(s, \chi_{n})$ and $L(s, \chi_{Dn})$ will be ineffective, as $n$ could be as large as $X$, and transforming the original summation length from $X$ to $n^2/X \asymp X^2/X = X$ shall not directly provide us with any saving in the estimation for $E$. We refer to more detail on this to the Section 1 and Section 4 of \cite{compMunshiMR2771124}.
On the other hand, upon taking absolute value to the innermost of $E$, Munshi showed, in \cite[Section 4-6]{compMunshiMR2771124},
\[E\ll X^{1/2}\int_{(1/2)}|\widetilde{J}(s)|\sum_{n\ll X^{1+\varepsilon}}\frac{|\lambda_f(n)|}{n^{1/2}}\left|W\left(\frac{n}{d\sqrt{q}}\right)\right||L(s,\chi_{Dn})L(s,\chi_n)|\,|ds|\ll X^{\frac{1}{2}}\cdot X^{\frac{1}{2}}(\log X)^{\frac{15}{2}}.\]

However, we notice that when working with a shorter range for $n$ of length $Y \ll X$, one may bound the truncated $E$ (denoted temporarily by $E(Y)$) by $o(X)$.
More precisely, following Munshi's method, one can show: for any $Y\ll X,$
\[E(Y)\ll X^{1/2}\int_{(1/2)}|\widetilde{J}(s)|\sum_{n\ll Y^{1+\varepsilon}}\frac{|\lambda_f(n)|}{n^{1/2}}\left|W\left(\frac{n}{d\sqrt{q}}\right)\right||L(s,\chi_{Dn})L(s,\chi_n)|\,|ds|\ll X^{\frac{1}{2}}\cdot Y^{\frac{1}{2}}(\log X)^{\frac{15}{2}}.\]

This motivates us to divide the $n$-summation in the approximate functional equation into two ranges:
$$
n\ll Y \quad \text{and} \quad n\gg Y,
$$
where $Y=\frac{X}{\log^{100}X}$.

For the first range, as discussed above, we may apply Munshi's method to extract the main term in Theorem~\ref{thm.main theorem} while controlling the error term.

For the second range, we employ Li's optimal estimate for bilinear forms of Hecke eigenvalues twisted by quadratic characters \cite[Proposition 3.2]{LiMR4768632}. It is worth noting that a direct application of Li's large sieve inequality would produce an error term exceeding the main term by a factor of $(\log X)^{\varepsilon}$. To overcome this difficulty, we adopt a more delicate argument inspired by the proof of \cite[Proposition~4.3]{zhou2025momentderivativesquadratictwists}. More precisely, we insert a smooth dyadic partition in $N$ over the second range and treat the three cases $N>X$, $Y\le N<X$, and $N<Y$ separately. In the first two cases, we apply the Cauchy--Schwarz inequality together with Li's estimation for bilinear forms. In the last case, we return to Munshi's method. Indeed, since $n\in[N,2N]$ and $N\le Y$, the resulting sum is again supported on $n\ll Y$, the range in which Munshi's argument is applicable.

\subsection{Notations} In this paper, the notation $\psum$ is the summation over all positive squarefree integers. The notation $\ssum$ is the dyadic summation. 

\section{Preliminaries}

\subsection{The weight function}

For the weight function $\psi(m,D),$ we have the following facts:
\begin{enumerate}
	\item $\psi(0,D)=0$ and $\psi(-m,D)=\psi(m,D).$
	\item If $m$ is a squarefree integer, then
	\[\psi(m,D)=\prod_{p|m}(1+\chi_D(p)).\]
	Moreover, assume that $m$ is squarefree. Then $\psi(m,D)\neq 0$ if and only if each $p|m$ is split in $\bQ(\sqrt{D}).$ In this case, $\psi(m,D)=\tau(d),$  the divisor function.
	\end{enumerate}
	
We prove the following lemma:
\begin{lemma}\label{lem. Hecke for psi}
	Let $m$ and $n$ be positive integers. Then 
	\[\psi(m,D)\psi(n,D)=\sum_{\ell|(m,n)}\chi_D(\ell)\psi\left(\frac{mn}{\ell^2},D\right).\]
	This implies, by M\"obius inversion,
	\[\psi(mn,D)=\sum_{\ell|(m,n)}\mu(\ell)\chi_D(\ell)\psi\left(\frac{m}{\ell},D\right)\psi\left(\frac{n}{\ell},D\right)\]
\end{lemma}
\begin{proof}
	It suffices to show, for $(mn,D)=1,$
	\[\psi(m,D)\psi(n,D)=\sum_{\ell|(m,n)}\chi_D(\ell)\psi\left(\frac{mn}{\ell^2},D\right).\]
	Noticing that $\psi(m,D)$ is a multiplicative function, it suffices to show that, for any prime $p$ satisfying $(p,D)=1$ and $\alpha,\beta$ being nonnegative integers,
	\[\psi(p^{\alpha},D)\psi(p^{\beta},D)=\sum_{\ell|(p^{\alpha},p^{\beta})}\chi_D(\ell)\psi\left(\frac{p^{\alpha+\beta}}{\ell^2},D\right).\]
	Without loss of generality, we can assume that $\alpha\leq\beta.$ We consider the following two cases:
	
	(a) if $\chi_D(p)=1,$ then $\psi(p^{\alpha},D)=\alpha+1$ and hence $\psi(p^{\alpha},D)\psi(p^{\beta},D)=(\alpha+1)(\beta+1).$ On the other hand:
	\[\sum_{\ell|(p^{\alpha},p^{\beta})}\chi_D(\ell)\psi\left(\frac{p^{\alpha+\beta}}{\ell^2},D\right)=\sum_{\ell=0}^{\alpha}(\alpha+\beta-2\ell+1)=(\alpha+1)(\beta+1).\]
	(b) if $\chi_D(p)=-1,$ then $\psi(p^{\alpha},D)=\frac{1+(-1)^{\alpha}}{2}$ and hence $\psi(p^{\alpha},D)\psi(p^{\beta},D)=\frac{1+(-1)^{\alpha}}{2}\cdot\frac{1+(-1)^{\beta}}{2}.$ On the other hand:
	\[\sum_{\ell|(p^{\alpha},p^{\beta})}\chi_D(\ell)\psi\left(\frac{p^{\alpha+\beta}}{\ell^2},D\right)=\sum_{\ell=0}^{\alpha}(-1)^{\ell}\frac{1+(-1)^{\alpha+\beta-2\ell}}{2}=\frac{1+(-1)^{\alpha}}{2}\cdot\frac{1+(-1)^{\beta}}{2}.\]
\end{proof}

\subsection{Approximate functional equation}\label{subsec. afe}
Let $f$ be a holomorphic newform of weight $2$ and (squarefree) level $q$. Assume that the global root number of $f$ is $\omega_f=-1.$ Let $d\equiv1\Mod{4}$ be a fundamental discriminant satisfying $(d,q)=1$ and $\chi_d(q)=1$ Then, by the approximate fucntional equation,
\[L'(1/2,f\times\chi_d)=(1+\chi_d(q))\sum_{n\geq1}\frac{\lambda_f(n)\chi_d(n)}{n^{1/2}}W\left(\frac{n}{d\sqrt{q}}\right),\]
where 
\begin{equation}\label{eq. W function}
W(y) := \frac{1}{2\pi i}\int_{(3)}\Gamma(1+w)(2\pi y)^{-w}\frac{dw}{w^2} = \frac{1}{2\pi i}\int_{(3)}\Gamma(w)(2\pi y)^{-w}\frac{dw}{w}.	
\end{equation}
We set $Y = X/(\log X)^{100}$. Then we define:
\[
	\mathcal{A}(d;Y) := (1 - \omega(f)\chi_d(q))\sum_{n\geq 1}\frac{\lambda_f(n)\chi_d(n)}{\sqrt{n}}W\left(\frac{n}{Y}\right)=(1 +\chi_d(q)) \sum_{n\geq 1}\frac{\lambda_f(n)\chi_d(n)}{\sqrt{n}}W\left(\frac{n}{Y}\right)
\]
and 
\[
	\mathcal{B}(d;Y) := L'(1/2, f\times \chi_d) - \mathcal{A}(d;Y).
\]
\subsection{Some special Dirichlet series}

In this section, we study some special Dirichlet series, which show up in later sections. We first recall the notations: let $q$ be a squarefree integer and $D<0$ be a fundamental discriminant satisfying $4|D$. Let $f$ be a holomorphic cusp form of weight $2$ and level $q$ satisfying $\omega_f=-1.$ 

\begin{lemma}\label{lem. euler for positive q prime}
	Assume the notations above. Let $q'\in\{4,4q\}.$ The following series:
	\[\cA_1(s;q',D):=\sum_{\substack{nq'=\square}}\frac{\lambda_f(n)}{n^{1/2+s}}\prod_{\substack{p|nq'\\p\nmid qD}}\frac{1-\chi_D(p)p^{-1}}{1+p^{-1}-(1+\chi_D(p))p^{-2}}\]
    has a Euler product expansion:
	\[\cA_1(s;q',D)=L(1+2s,\sym^2f)\cH_1(s,q',D),\]
	where $\cH_1(s,q',D)$ is a Euler product, absolutely convergent when $\Re(s)>-\frac{1}{4}.$ Moreover, if $(q,6)=1$, $\cH_1(s,q',D)>0$ for both $q'=4$ and $q'=4q.$
\end{lemma}
\begin{proof}
	We first assume that $\Re(s)\gg1.$ Since $4|D,$ we have $q'|qD.$ We set $q_0'$ is the squarefree part of $q',$ that is, $q_0'=1$ if $q'=4$ and $q_0'=q$ is $q'=4q.$ Then
\begin{flalign*}
\cA_1(s;q',D)&=\sum_{\substack{nq'=\square}}\frac{\lambda_f(n)}{n^{1/2+s}}\prod_{\substack{p|n\\p\nmid qD}}\frac{1-\chi_D(p)p^{-1}}{1+p^{-1}-(1+\chi_D(p))p^{-2}}\\
	&=\prod_{p|qD}\left(\sum_{\ell\geq0}\frac{\lambda_f(p^{2\ell+e_p(q_0')})}{p^{(2\ell+e_p(q_0'))(1/2+s)}}\right)\prod_{p\nmid qD}\left(1+\frac{1-\chi_D(p)p^{-1}}{1+p^{-1}-(1+\chi_D(p))p^{-2}}\left(\sum_{\ell\geq1}\frac{\lambda_f(p^{2\ell})}{p^{\ell(1+2s)}}\right)\right)\\
	&=\prod_{p|qD}\left(\sum_{\ell\geq0}\frac{\lambda_f(p^{2\ell+e_p(q_0')})}{p^{(2\ell+e_p(q_0'))(1/2+s)}}\right)\prod_{\substack{p\nmid qD\\\chi_D(p)=-1}}\left(\sum_{\ell\geq0}\frac{\lambda_f(p^{2\ell})}{p^{\ell(1+2s)}}\right)\prod_{\substack{p\nmid qD\\\chi_D(p)=1}}\left(1+\frac{1}{1+2p^{-1}}\left(\sum_{\ell\geq1}\frac{\lambda_f(p^{2\ell})}{p^{\ell(1+2s)}}\right)\right)\\
	&=L(1+2s,\sym^2f)\cH_1(s,q',D)
	\end{flalign*}
	$\cH_1(s,q',D)$ is absolutely convergent for $\Re(s)>-\frac{1}{4}.$ By Rankin-Selberg theory, we have:
	\[L_p(s,\sym^2f)=\zeta_p(2s)\left(\sum_{\ell\geq0}\frac{\lambda_f(p^{2\ell})}{p^{\ell s}}\right)\]
	 for $p\nmid q.$ By Atkin-Lehner theory, $\lambda_f(p)^{\ell}=\lambda_f(p^{\ell})$ for any $p|q$ and $\ell\geq0.$ This will imply: if we set $\cH_1(s,q',D)=\prod_{p}\cH_{1,p}(s,q',D),$ then
	 \[\cH_{1,p}(s,4,D)=\begin{cases}
	 	1&\mbox{if $p|q$}\\
	 	\zeta_p(2+4s)^{-1}&\mbox{if $p|D$}\\
	 	\zeta_p(2+4s)^{-1}&\mbox{if $p\nmid qD$ and $\chi_D(p)=-1$}\\
	 	\zeta_p(2+4s)^{-1}+\frac{2}{p+2}\left(L_p(1+2s,\sym^2f)^{-1}-\zeta_p(2+4s)^{-1}\right)&\mbox{if $p\nmid qD$ and $\chi_D(p)=1$}\\
	 \end{cases}\]
	 and
	 \[\cH_{1,p}(s,4q,D)=\begin{cases}
	 	\lambda_f(p)p^{-1/2-s}&\mbox{if $p|q$}\\
	 	\cH_{1,p}(s,4,D)&\mbox{otherwise}
	 \end{cases}\]
This yields the following relation:	 
	 \[\cH_1(s,4q,D)=\frac{\lambda_f(q)}{q^{1/2+s}}\cH_1(s,4,D).\]
By checking the local factors, $\cH_1(0,q',D)\neq0$ provided that $(q,6)=1.$ In this case, $\cH_1(0,4,D)>0.$ Then $\cH_1(0,4q,D)$ has the same sign with $\lambda_f(q).$
	This is positive since $\lambda_f(q)=-\frac{\omega_f}{\sqrt{q}}>0.$
	\end{proof}

For $q'\in\{-4,-4q\},$ we have the following similar lemma:
\begin{lemma}\label{lem. euler for negative q prime}
	Assume the notations above. Let $q'\in\{-4,-4q\}.$ The following series:
	\[\cA_2(s;q',D):=\sum_{\substack{Dnq'=\square}}\frac{\lambda_f(n)}{n^{1/2+s}}\prod_{\substack{p|nq'\\p\nmid qD}}\frac{1-\chi_D(p)p^{-1}}{1+p^{-1}-(1+\chi_D(p))p^{-2}}\]
    has a Euler product expansion:
	\[\cA_2(s;q',D)=L(1+2s,\sym^2f)\cH_2(s,q',D),\]
	where $\cH_1(s,q',D)$ is a Euler product, absolutely convergent when $\Re(s)>-\frac{1}{4}.$ Moreover, if $(q,6)=1$, $\cH_2(s,q',D)\neq0$ for both $q'=-4$ and $q'=-4q.$
\end{lemma}
\begin{proof}
	We set $D_0$ to be the squarefree part of $D.$ Since $4|D$ and $D$ is a fundamental discriminant, $D_0=|D|/4.$ We again set $q_0'$ the squarefree part of $q'.$ Then $q_0'=|q'|/4$ since $q'\in\{-4,-4q\},$ $q>0$ and $q$ is squarefree. Similar to the proof of Lemma \ref{lem. euler for positive q prime}, we can show:
\begin{flalign*}
\cA_2(s;q',D)&=\prod_{p|qD}\left(\sum_{\ell\geq0}\frac{\lambda_f(p^{2\ell+e_p(q_0'D_0)})}{p^{(2\ell+e_p(q_0'D_0))(1/2+s)}}\right)\prod_{\substack{p\nmid qD\\\chi_D(p)=-1}}\left(\sum_{\ell\geq0}\frac{\lambda_f(p^{2\ell})}{p^{\ell(1+2s)}}\right)\prod_{\substack{p\nmid qD\\\chi_D(p)=1}}\left(1+\frac{1}{1+2p^{-1}}\left(\sum_{\ell\geq1}\frac{\lambda_f(p^{2\ell})}{p^{\ell(1+2s)}}\right)\right)\\
	&=L(1+2s,\sym^2f)\cH_2(s,q',D).
	\end{flalign*}		
We notice that, for $p\nmid q$
\[\left(1+\frac{1}{p^{2s}}\right)\left(\sum_{\ell\geq0}\frac{\lambda_f(p^{2\ell+1})}{p^{(2\ell+1)s}}\right)=\frac{\lambda_f(p)}{p^{s}}\left(\sum_{\ell\geq0}\frac{\lambda_f(p^{2\ell})}{p^{2\ell s}}\right)=\frac{\lambda_f(p)}{p^s}\frac{L_p(2s,\sym^2f)}{\zeta_p(4s)},\]	
which implies:
\[\sum_{\ell\geq0}\frac{\lambda_f(p^{2\ell+1})}{p^{(2\ell+1)s}}=\frac{\lambda_f(p)}{p^s}\frac{L_p(2s,\sym^2f)}{\zeta_p(2s)}.\]
	This yields, if we write $\cH_{2}(s,q',D)=\prod_{p}\cH_{2,p}(s,q',D)$,
		 \[\cH_{2,p}(s,-4,D)=\begin{cases}
	 	1&\mbox{if $p|q$}\\
	 	\lambda_f(p)p^{-1/2}\zeta_p(1+2s)^{-1}&\mbox{if $8|D$ and $p|D$ or $4\mid\!\mid D$ and $p\neq2$}\\
	 	\zeta_p(2+4s)^{-1}&\mbox{if $4\mid\!\mid D$ and $p=2$}\\
	 	\zeta_p(2+4s)^{-1}&\mbox{if $p\nmid qD$ and $\chi_D(p)=-1$}\\
	 	\zeta_p(2+4s)^{-1}+\frac{2}{p+2}\left(L_p(1+2s,\sym^2f)^{-1}-\zeta_p(2+4s)^{-1}\right)&\mbox{if $p\nmid qD$ and $\chi_D(p)=1$}\\
	 \end{cases}\]
	 and
	 \[\cH_{2,p}(s,-4q,D)=\begin{cases}
	 	\lambda_f(p)p^{-1/2-s}&\mbox{if $p|q$}\\
	 	\cH_{2,p}(s,4,D)&\mbox{otherwise}
	 \end{cases}.\]
	 This yields the following relation:	 
	 \[\cH_2(s,-4q,D)=\frac{\lambda_f(q)}{q^{1/2+s}}\cH_2(s,-4,D)\]
	 and
	 \[\cH_2(s,-4,D)=\frac{\lambda_f(D_0)}{D_0^{1/2+s}}\prod_{p|D_0}\frac{1}{1+p^{-1-2s}}\cH_1(s,4,D).\]
\end{proof}

\begin{remark}\label{rem. constant is positive}
	Let $D_0$ be the squarefree part of $D.$ By the relations in Lemma \ref{lem. euler for positive q prime} and Lemma \ref{lem. euler for negative q prime}, we can show:
	 \begin{flalign*}
	 	\cH_1(0,4,D)+\cH_1(0,4q,D)+\cH_2(0,-4,D)+\cH_2(0,-4q,D)&=\left(1+\frac{\lambda_f(D_0)}{D_0\prod_{p|D_0}(1+p^{-1})}\right)\left(1+\frac{\lambda_f(q)}{q^{1/2}}\right)\cH_1(0,4,D).
	 \end{flalign*}
	We can easily check, by Delinge's bound, that 
	\[\left|\frac{\lambda_f(D_0)}{D_0^{1/2}\prod_{p|D_0}(1+p^{-1})}\right|=\prod_{p|D_0}\left|\frac{\lambda_f(p)}{p^{1/2}+p^{-1/2}}\right|<1.\] 
	Therefore, the summation above is positive when $(q,6)=1$.
\end{remark}

\subsection{Roadmap to the proof of Theorem \ref{thm.main theorem}}
In this section, we prove Theorem \ref{thm.main theorem} assuming two key propositons. We recall the notations in \S \ref{subsec. afe}. 

Then Theorem \ref{thm.main theorem} follows form the Propositions below.

\begin{prop}\label{prop.difference prop}
	With the notations above,
	\[\sum_{d\in\cD}\psi(d,D)\cB(d;Y)J\left(\frac{d}{X}\right)\ll X(\log X)^{\frac{1}{2}}(\log\log X)^3.\]
\end{prop}

Proposition \ref{prop.difference prop} reduces Theorem \ref{thm.main theorem} to the analysis of the first moment of $\cA(d;Y),$ which will be investigated in the following proposition. Before the statement of the proposition, we introduce the following constants:
\begin{equation}\label{eq.constant in euler}
	c_{q,D}=\frac{(\cH_1(0,4,D)+\cH_1(0,4q,D)+\cH_2(0,-4,D)+\cH_2(0,-4q,D))}{2}
\end{equation}
and
\begin{equation}\label{eq. extra term}
	\cG(q,D)=\prod_{\substack{p\nmid qD}}\left(1-\frac{1}{p^2}\left(\psi(p^2;D)-\frac{\chi_D(p)\psi(p;D)}{p}\right)\right)\prod_{\substack{p|qD}}\left(1-\frac{1}{p}\left(\psi(p;D)-\frac{\chi_D(p)}{p}\right)\right).
\end{equation}
For $q'\in\{\pm4,\pm 4q\}$ and $i\in\{1,2\}$, the functions $\cH_i(s,q',D)$ are defined in Lemma \ref{lem. euler for positive q prime} and Lemma \ref{lem. euler for negative q prime}. When $(q,6)=1$, $c_{q,D}>0$ by Remark \ref{rem. constant is positive}. We can also show that $\cG(q,D)>0.$

\begin{prop}\label{prop.truncated sum}
	With the notations above,
	\begin{flalign*}
		\sum_{d\in\cD}\psi(d,D)\cA(d;Y)J\left(\frac{d}{X}\right)=c_{q,D}\cG(q,D)\widetilde{J}(1)L(1,\chi_D)X\log X+O(X\log\log X),
	\end{flalign*}
	where $c_{q,D}$ (resp. $\cG(q,D)$) is defined in \eqref{eq.constant in euler} (resp. \eqref{eq. extra term}).
\end{prop}

We will prove Proposition \ref{prop.truncated sum} (resp. Proposition \ref{prop.difference prop}) in \S \ref{sec. proof of prop 2} (resp. \S \ref{sec. proof of prop 1}). Then the proof of Theorem \ref{thm.main theorem} is completed if we combine Proposition \ref{prop.difference prop} and Proposition  \ref{prop.truncated sum}.

\section{Proof of Proposition \ref{prop.truncated sum}}\label{sec. proof of prop 2}

In this section, we study:
\[
    S(X)=\sum_{d\in \mathcal{D}}\psi(d;D)\mathcal{A}(d;Y)J\left(\frac{d}{X}\right). 
\]

First, we have, noticing that $\chi_d(q)=-1$ will imply that $\cA(d;Y)=0,$:
\begin{flalign*}
S(X)&=\psum_{\substack{(d,Dq)=1 \\ d \equiv 1 \Mod{4}}}\psi(d;D)\mathcal{A}(d;Y)J\left(\frac{d}{X}\right)=\sum_{\substack{(d,Dq)=1 \\ d \equiv 1 \Mod{4}}}\psi(d;D)\mathcal{A}(d;Y)J\left(\frac{d}{X}\right)\sum_{a^2|d}\mu(a)\\
&=\sum_{(a,2Dq)=1}\mu(a)\sum_{\substack{(d,Dq)=1 \\ d \equiv 1 \Mod{4}}}\psi(a^2d;D)J\left(\frac{a^2d}{X}\right)(1 +\chi_d(q))
\sum_{(n,a)=1}\frac{\lambda_f(n)\chi_{d}(n)}{\sqrt{n}}
	W\left(\frac{n}{Y}\right)
	\end{flalign*}.
	
We set $Z=\log^{20} X.$ Then we write:
\begin{equation}\label{eq. li truncation}
	S(X)=\fS_1+\fS_2
\end{equation}
where 
\[\fS_1=\sum_{\substack{a\leq Z\\(a,2qD)=1}}\mu(a)\sum_{\substack{(d,qD)=1\\d\equiv1\Mod{4}}}\psi(a^2d,D)J\left(\frac{a^2d}{X}\right)(1 +\chi_d(q))\sum_{(n,a)=1}\frac{\lambda_f(n)\chi_d(n)}{\sqrt{n}}W\left(\frac{n}{Y}\right)\]
and
\[\fS_2=\sum_{\substack{a> Z\\(a,2qD)=1}}\mu(a)\sum_{\substack{(d,qD)=1\\d\equiv1\Mod{4}}}\psi(a^2d,D)J\left(\frac{a^2d}{X}\right)(1 +\chi_d(q))\sum_{(n,a)=1}\frac{\lambda_f(n)\chi_d(n)}{\sqrt{n}}W\left(\frac{n}{Y}\right).\]
The following lemma will show that the $a>Z$ part will only contribute the error term:
\begin{lemma}\label{lem. truncated terms}
	Assume the notations as above. Then
	\[\fS_2=\sum_{\substack{a> Z\\(a,2qD)=1}}\mu(a)\sum_{\substack{(d,qD)=1\\d\equiv1\Mod{4}}}\psi(a^2d,D)J\left(\frac{a^2d}{X}\right)(1 +\chi_d(q))\sum_{(n,a)=1}\frac{\lambda_f(n)\chi_d(n)}{\sqrt{n}}W\left(\frac{n}{Y}\right)\ll\frac{X(\log X)^{3}}{Z^{1-\varepsilon}}.\]
\end{lemma}
\begin{proof}

	We apply Cauchy-Schwarz inequality for the $d$-sum:
	\begin{equation}\label{eq. CS for large a contribution}
		\fS_2\ll\sum_{\substack{a> Z\\(a,qD)=1}}\left(\sum_{d\geq1}\psi(a^2d,D)^2J\left(\frac{a^2d}{X}\right)\right)^{1/2}\left(\sum_{(d,2)=1}J\left(\frac{a^2d}{X}\right)\left|\sum_{(n,a)=1}\frac{\lambda_f(n)\chi_d(n)}{\sqrt{n}}W\left(\frac{n}{Y}\right)\right|^2\right)^{1/2}
	\end{equation}
	We can show that, by Lemma \ref{lem. Hecke for psi},
\[\psi(a^2d)=\sum_{\ell|(a^2,d)}\mu(\ell)\psi\left(\frac{a^2}{\ell},D\right)\psi\left(\frac{d}{\ell},D\right)\ll \sum_{\ell|(a^2,d)}\tau\left(\frac{a^2}{\ell}\right)\psi\left(\frac{d}{\ell},D\right)\ll a^{\varepsilon}\sum_{\ell|d}\psi\left(\frac{d}{\ell},D\right).\]
By Landau-Selberg-Delange theorem, this implies:
	\[\sum_{d\geq1}\psi(a^2d,D)^2J\left(\frac{a^2d}{X}\right)\ll \frac{X(\log X)^5}{a^{2-\varepsilon}}.\]
	Similar to the argument in \cite[Lemma 6.3]{LiMR4768632}, we can show:
	\[\sum_{(d,2)=1}J\left(\frac{a^2d}{X}\right)\left|\sum_{(n,a)=1}\frac{\lambda_f(n)\chi_d(n)}{n^{1/2}}W\left(\frac{n}{Y}\right)\right|^2\ll \frac{\tau(a)^5}{a^2}X\log X.\]
	We insert the bounds into \eqref{eq. CS for large a contribution}, and we complete the proof.
\end{proof}

By the choice of $Z$, it suffices to study $\fS_1.$ Noticing that $d>0$ and $d\equiv1\Mod{4},$ we apply $\chi_d(q)=\chi_q(d)$ and $\chi_d(n)=\chi_n(d).$ We use 
\[\frac{\chi_{4}(d)+\chi_{-4}(d)}{2}=\begin{cases}
	1& \mbox{if $d\equiv 1\Mod{4}$}\\
	0& \mbox{if $d\equiv 3\Mod{4}$}\\
\end{cases},\]
which implies:
\begin{flalign*}
	\fS_1&=\sum_{\substack{a\leq Z\\(a,2qD)=1}}\mu(a)\sum_{(d,2qD)=1}\psi(a^2d,D)J\left(\frac{a^2d}{X}\right)\frac{(\chi_4(d)+\chi_{-4}(d))(1 +\chi_q(d))}{2}\sum_{(n,a)=1}\frac{\lambda_f(n)\chi_n(d)}{\sqrt{n}}W\left(\frac{n}{Y}\right)\\
	&=\frac{1}{2}\sum_{q'\in\{\pm 4,\pm 4q\}}\sum_{\substack{a\leq Z\\(a,2qD)=1}}\mu(a)\sum_{(d,2qD)=1}\psi(a^2d,D)J\left(\frac{a^2d}{X}\right)\sum_{(n,a)=1}\frac{\lambda_f(n)\chi_{nq'}(d)}{\sqrt{n}}W\left(\frac{n}{Y}\right)
\end{flalign*}
We then insert the M\"obius function to drop the coprime conidtion in $d$:
\[\fS_1=\frac{1}{2}\sum_{q'\in\{\pm 4,\pm 4q\}}S_{q'}(a\leq Z),\]
where
\[S_{q'}(a\leq Z)=\sum_{\substack{a\leq Z\\(a,2Dq)=1}}\mu(a)\sum_{b|Dq}\mu(b)\sum_{(n,a)=1}\frac{\lambda_f(n)\chi_{nq'}(b)}{\sqrt{n}}
	W\left(\frac{n}{Y}\right)\sum_{d}\psi(a^2bd;D)\chi_{nq'}(d)J\left(\frac{a^2bd}{X}\right).\]
We recall \eqref{eq. li truncation} and Lemma \ref{lem. truncated terms}, which imply:
\begin{equation}\label{eq. moment reduction 1}
	S(X)=\frac{1}{2}\sum_{q'\in\{\pm 4,\pm 4q\}}S_{q'}(a\leq Z)+O\left(\frac{X(\log X)^{3}}{Z^{1-\varepsilon}}\right).
\end{equation}
Therefore, it suffices to study $S_{q'}(a\leq Z)$.

Utilizing Lemma \ref{lem. Hecke for psi}, we obtain:
\begin{align*}
	S_{q'}(a\leq Z) & = \sum_{\substack{a \leq Z \\ (a,2qD)=1}}\mu(a)\sum_{b|qD}\mu(b)\sum_{\ell|a^2b}\mu(\ell)\chi_D(\ell)\psi\left(\frac{a^2b}{\ell};D\right)\\
	&\hspace{10mm} \times \sum_{(n,a)=1}\frac{\lambda_f(n)}{\sqrt{n}}W\left(\frac{n}{Y}\right)\chi_{nq'}(b\ell)\sum_{d\geq 1}\psi(d;D)\chi_{nq'}(d)J\left(\frac{a^2bd\ell}{X}\right). 
\end{align*}
By invoking the Mellin inversion of $J(\cdot)$, we have 
\begin{equation}\label{eq. inner-d-sum}
    \sum_{d\geq 1}\psi(d;D)\chi_{nq'}(d)J(la^2bd/X) = \frac{1}{2\pi i}\int_{(2)}
	\tilde{J}(s)\left(\frac{X}{a^2bl}\right)^sL(s, \chi_{Dnq'})L(s,\chi_{nq'})ds.
\end{equation}
We insert \eqref{eq. inner-d-sum} into $S_{q'}(a\leq Z),$ which implies:
\begin{equation}\label{eq. small a term}
	S_{q'}(a\leq Z)  = \sum_{\substack{a \leq Z \\ (a,2qD)=1}}\mu(a)\sum_{b|qD}\mu(b)\sum_{\ell|a^2b}\mu(\ell)\chi_D(\ell)\psi\left(\frac{a^2b}{\ell};D\right)\cS_{q',a,b,\ell}(X),
\end{equation}
where
\[\cS_{q',a,b,\ell}(X)=\frac{1}{2\pi i}\int_{(2)}
	\tilde{J}(s)\left(\frac{X}{a^2bl}\right)^s\sum_{(n,a)=1}\frac{\lambda_f(n)\chi_{nq'}(b\ell)}{\sqrt{n}}W\left(\frac{n}{Y}\right)L(s, \chi_{Dnq'})L(s,\chi_{nq'})ds.\]
We shift the contour to $\Re(s)=\frac{1}{2}.$ Noticing that $L(s,\chi_n)$ has a simple at $s=1$ if and only if $n$ is a square, we obtain:
\begin{equation}\label{eq. mellin term}
\cS_{q',a,b,\ell}(X)=M(q')+E
\end{equation} 
where
\begin{equation}\label{eq. main term}
\begin{split}
M(q'):=M_{q',a,b,\ell}(X)=&\frac{X\widetilde{J}(1)}{a^2b\ell}\sum_{(n,a)=1}\frac{\lambda_f(n)\chi_{nq'}(b\ell)}{\sqrt{n}}W\left(\frac{n}{Y}\right)\mathbf{1}_{Dnq'=\square}L(1,\chi_{nq'})\prod_{p|Dnq'}\left(1-\frac{1}{p}\right)\\
&+\frac{X\widetilde{J}(1)}{a^2b\ell}\sum_{(n,a)=1}\frac{\lambda_f(n)\chi_{nq'}(b\ell)}{\sqrt{n}}W\left(\frac{n}{Y}\right)\mathbf{1}_{nq'=\square}L(1,\chi_{Dnq'})\prod_{p|nq'}\left(1-\frac{1}{p}\right)\\	
\end{split}	
\end{equation}
and
\[E:=E_{q',a,b,\ell}(X)=\frac{1}{2\pi i}\int_{(1/2)}
	\tilde{J}(s)\left(\frac{X}{a^2bl}\right)^s\sum_{(n,a)=1}\frac{\lambda_f(n)\chi_{nq'}(b\ell)}{\sqrt{n}}W\left(\frac{n}{Y}\right)L(s, \chi_{Dnq'})L(s,\chi_{nq'})ds.\]
\subsection{The error term $E$}\label{subsec. error} We can bound $E$ trivially by
\[E\ll\left(\frac{X}{a^2b\ell}\right)^{\frac{1}{2}}\int_{(1/2)}\left|\widetilde{J}(s)\right|\sum_{n\ll Y^{1+\varepsilon}}\frac{|\lambda_f(n)|}{\sqrt{n}}W\left(\frac{n}{Y}\right)|L(s,\chi_{Dnq'})L(s,\chi_{nq'})| |\,ds|.\]
For $N\ll Y^{1+\varepsilon}$ and $n\in[N,2N],$ we can show
\begin{equation}\label{eq. growth of W}
	W\left(\frac{n}{Y}\right)\ll H(N)=\begin{cases}
	\log Y &\mbox{if $N\leq Y$}\\
	\frac{Y}{N}&\mbox{if $Y\leq N\leq Y^{1+\varepsilon}$}
\end{cases}.
\end{equation}
Following \cite[Equation (14)-(17)]{compMunshiMR2771124}, we obtain
\[E\ll\left(\frac{X}{a^2b\ell}\right)^{\frac{1}{2}}\dsum_{N}\frac{H(N)}{\sqrt{N}}\int_{\bR}U(N,t)\frac{\,dt}{(1+|t|)^6},\]
where
\[U(N,t)=\sum_{N\leq n<2N}|\lambda_f(n)|\left(\left|L\left(\frac{1}{2}+it,\chi_{4n}\right)\right|^2+\left|L\left(\frac{1}{2}+it,\chi_{-4n}\right)\right|^2\right).\]
Similar to the argument in \cite[Section 4]{compMunshiMR2771124}, we replace $\chi_{nq'}$ (resp. $\chi_{Dnq'}$) by $\chi_{4n}$ (resp. $\chi_{-4n}$) by the positivity and the fact that $D,q'$ are fixed statisfying $4|q'.$ 

Similar to the argument in \cite[Section 4-6]{compMunshiMR2771124}, we can show that (see \cite[page~32]{compMunshiMR2771124})
\[U(N,t)\ll N(1+|t|)^3(\log N)^{\frac{9}{2}}.\]
We insert it into $E,$ which implies:
\[E\ll \left(\frac{X}{a^2b\ell}\right)^{\frac{1}{2}}\dsum_{N}H(N)\sqrt{N}(\log N)^{\frac{9}{2}}\ll\frac{(XY)^{\frac{1}{2}}\log^6 Y}{(a^2b\ell)^{\frac{1}{2}}}.\]

We then insert $\cS_{q',a,b,\ell}(X)$ into \eqref{eq. small a term}, which implies:
\begin{equation}\label{eq. separation of main and error for S}
	S_{q'}(a\leq Z)=\cM(q')+O\left((XY)^{\frac{1}{2}}\log^6Y\sum_{a\leq Z}\sum_{b|qD}\sum_{\ell|a^2b}\frac{\psi(a^2b/\ell,D)}{(a^2b\ell)^{1/2}}\right),
\end{equation}
where
\[\cM(q'):=\sum_{\substack{a \leq Z \\ (a,2qD)=1}}\mu(a)\sum_{b|qD}\mu(b)\sum_{\ell|a^2b}\mu(\ell)\chi_D(\ell)\psi\left(\frac{a^2b}{\ell};D\right)M(q')\]
with $M(q')$ defined in \eqref{eq. main term}. By the choice of $Y$ and $Z$, the second term will only contribute the error term of the size $O\left(\frac{X}{\log^{20} X}\right).$ 

We insert \eqref{eq. separation of main and error for S} into \eqref{eq. moment reduction 1}, which implies:
\begin{equation}\label{eq. main for moment}
	S(X)=\frac{1}{2}\sum_{q'\in\{\pm 4,\pm 4q\}}\cM(q')+O\left(\frac{X(\log X)^{3}}{Z^{1-\varepsilon}}\right)+O\left(\frac{X}{\log^{20} X}\right).
\end{equation}
Therefore, it suffices to study $\cM(q')$ in the rest of the section.

\subsection{The term $\cM(q')$}\label{subsec. main term} We recall that
\[\cM(q')=\sum_{\substack{a \leq Z \\ (a,2qD)=1}}\mu(a)\sum_{b|qD}\mu(b)\sum_{\ell|a^2b}\mu(\ell)\chi_D(\ell)\psi\left(\frac{a^2b}{\ell};D\right)M(q'),\]
where $M(q')$ is defined in \eqref{eq. main term}. We consider two separated cases: $q'>0$ and $q'<0.$

\noindent\textbf{Case $q'>0$}: If $q'>0,$ then $Dnq'$ can never be a square since $D<0$. This yields (noticing that $(a,2q)=1$):
\[M(q')=\frac{X\widetilde{J}(1)L(1,\chi_D)}{a^2b\ell}\sum_{\substack{nq'=\square\\(nq',ab)=1}}\frac{\lambda_f(n)}{\sqrt{n}}W\left(\frac{n}{Y}\right)\prod_{p|nq'}\left(1-\frac{1}{p}\right)\left(1-\frac{\chi_D(p)}{p}\right).\]
We notice that $\left|\prod_{p|nq'}\left(1-\frac{1}{p}\right)\left(1-\frac{\chi_D(p)}{p}\right)\right|\leq 1$. This will imply:
 \[M(q')\ll \frac{X\log^2 X}{a^2b\ell}\]
and hence
\[\cM(q')=\sum_{\substack{a\geq1 \\ (a,2qD)=1}}\mu(a)\sum_{b|qD}\mu(b)\sum_{\ell|a^2b}\mu(\ell)\chi_D(\ell)\psi\left(\frac{a^2b}{\ell};D\right)M(q')+O\left(\frac{X\log^2 X}{Z^{1-\varepsilon}}\right).\]
We set
\begin{flalign*}
\cM_1(q'):&=\sum_{\substack{a\geq1 \\ (a,2qD)=1}}\mu(a)\sum_{b|qD}\mu(b)\sum_{\ell|a^2b}\mu(\ell)\chi_D(\ell)\psi\left(\frac{a^2b}{\ell};D\right)M(q')\\
&=X\widetilde{J}(1)L(1,\chi_D)\sum_{\substack{nq'=\square}}\frac{\lambda_f(n)}{\sqrt{n}}W\left(\frac{n}{Y}\right)\prod_{p|nq'}\left(1-\frac{1}{p}\right)\left(1-\frac{\chi_D(p)}{p}\right)\\ 
&\hspace{20mm}\times\sum_{\substack{a\geq1 \\ (a,2nqD)=1}}\frac{\mu(a)}{a^2}\sum_{\substack{b|qD\\(b,nq')=1}}\frac{\mu(b)}{b}\sum_{\ell|a^2b}\frac{\mu(\ell)\chi_D(\ell)\psi\left(\frac{a^2b}{\ell};D\right)}{\ell}.	
\end{flalign*}
We notice that the last summation define a multiplicative function. Then
\begin{flalign*}
\sum_{\substack{a\geq1 \\ (a,2nqD)=1}}&\frac{\mu(a)}{a^2}\sum_{\substack{b|qD\\(b,nq')=1}}\frac{\mu(b)}{b}\sum_{\ell|a^2b}\frac{\mu(\ell)\chi_D(\ell)\psi\left(\frac{a^2b}{\ell};D\right)}{\ell}\\
&=\sum_{\substack{a\geq1 \\ (a,2nqD)=1}}\frac{\mu(a)}{a^2}\sum_{\ell_1|a^2}\frac{\mu(\ell_1)\chi_D(\ell_1)\psi\left(\frac{a^2}{\ell_1};D\right)}{\ell_1}\sum_{\substack{b|qD\\(b,nq')=1}}\frac{\mu(b)}{b}\sum_{\ell_2|b}\frac{\mu(\ell_2)\chi_D(\ell_2)\psi\left(\frac{b}{\ell_2};D\right)}{\ell_2}\\
&=\prod_{\substack{p\nmid qD\\p\nmid nq'}}\left(1-\frac{1}{p^2}\left(\psi(p^2;D)-\frac{\chi_D(p)\psi(p;D)}{p}\right)\right)\prod_{\substack{p|qD\\ p\nmid nq'}}\left(1-\frac{1}{p}\left(\psi(p;D)-\frac{\chi_D(p)}{p}\right)\right)
\end{flalign*}
We insert it into $\cM_1(q')$ and open the $W$-function by \eqref{eq. W function}:
\begin{flalign*}
	\cM_1(q')=\frac{X\widetilde{J}(1)L(1,\chi_D)\cG(q,D)}{2\pi i}\int_{(3)}\left(\frac{Y}{2\pi}\right)^w\Gamma(1+w)\cA_1(w;q',D) \frac{dw}{w^2}
\end{flalign*}
where
\begin{flalign*}
	\cA_1(s;q',D):&=\sum_{\substack{nq'=\square}}\frac{\lambda_f(n)}{n^{1/2+s}}\prod_{\substack{p|nq'\\p\nmid qD}}\frac{(1-p^{-1})(1-\chi_D(p)p^{-1})}{1-p^{-2}(\psi(p^2;D)-\chi_D(p)\psi(p;D)p^{-1})}\prod_{\substack{p|nq'\\p|qD}}\frac{(1-p^{-1})(1-\chi_D(p)p^{-1})}{1-p^{-1}(\psi(p;D)-\chi_D(p)p^{-1})}\\
	&=\sum_{\substack{nq'=\square}}\frac{\lambda_f(n)}{n^{1/2+s}}\prod_{\substack{p|nq'\\p\nmid qD}}\frac{1-\chi_D(p)p^{-1}}{1+p^{-1}-(1+\chi_D(p))p^{-2}}.
\end{flalign*}
This coincides with the series in Lemma \ref{lem. euler for positive q prime}. We also recall that $\cG(q,D)$ is defined in \eqref{eq. extra term} by
\[\cG(q,D)=\prod_{\substack{p\nmid qD}}\left(1-\frac{1}{p^2}\left(\psi(p^2;D)-\frac{\chi_D(p)\psi(p;D)}{p}\right)\right)\prod_{\substack{p|qD}}\left(1-\frac{1}{p}\left(\psi(p;D)-\frac{\chi_D(p)}{p}\right)\right).\]
Then we shift the contour in $\cM_1(q')$ to $\Re(s)=-\frac{1}{8}$ and apply Lemma \ref{lem. euler for positive q prime}:
\[\cM_1(q')=\widetilde{J}(1)L(1,\chi_D)\cG(q,D)\cH_1(0,q',D)X\log X+O(X\log\log X),\]
and hence for $q'\in\{q,4q\}:$
\begin{equation}\label{eq. positive main}
	\cM(q')=\widetilde{J}(1)L(1,\chi_D)\cG(q,D)\cH_1(0,q',D)X\log X+O(X\log\log X).
\end{equation}

\noindent\textbf{Case $q'<0$}: If $q'<0,$ then $nq'$ can never be a square. This yields:
\[M(q')=\frac{X\widetilde{J}(1)}{a^2b\ell}\sum_{(n,a)=1}\frac{\lambda_f(n)\chi_{nq'}(b\ell)}{\sqrt{n}}W\left(\frac{n}{Y}\right)\mathbf{1}_{Dnq'=\square}L(1,\chi_{nq'})\prod_{p|Dnq'}\left(1-\frac{1}{p}\right).\]
When $Dnq'$ is a square, we can show that $D|nq'$. by our choice of $D$, $q$ and $q'$. This implies that $\chi_{nq'}$ is induced from $\chi_D$, which yields:
\[M(q')=\frac{X\widetilde{J}(1)L(1,\chi_D)\chi_D(b\ell)}{a^2b\ell}\sum_{\substack{Dnq'=\square\\(nq',ab)=1}}\frac{\lambda_f(n)}{\sqrt{n}}W\left(\frac{n}{Y}\right)\prod_{p|Dnq'}\left(1-\frac{\chi_D(p)}{p}\right)\left(1-\frac{1}{p}\right).\]
Following the argument in the $q'>0$ case, we obtain:
\[\cM(q')=\cM_2(q)+O\left(\frac{X\log^2 X}{Z^{1-\varepsilon}}\right),\]
where
\begin{flalign*}
	\cM_2(q')&=X\widetilde{J}(1)L(1,\chi_D)\sum_{\substack{Dnq'=\square}}\frac{\lambda_f(n)}{\sqrt{n}}W\left(\frac{n}{Y}\right)\prod_{p|Dnq'}\left(1-\frac{1}{p}\right)\left(1-\frac{\chi_D(p)}{p}\right)\\ 
&\hspace{20mm}\times\sum_{\substack{a\geq1 \\ (a,2nqD)=1}}\frac{\mu(a)}{a^2}\sum_{\substack{b|qD\\(b,nq')=1}}\frac{\mu(b)\chi_D(b)}{b}\sum_{\ell|a^2b}\frac{\mu(\ell)\chi_D(\ell)^2\psi\left(\frac{a^2b}{\ell};D\right)}{\ell}.
\end{flalign*}
Similarly, we show:
\begin{flalign*}
\sum_{\substack{a\geq1 \\ (a,2nqD)=1}}&\frac{\mu(a)}{a^2}\sum_{\substack{b|qD\\(b,nq')=1}}\frac{\mu(b)\chi_D(b)}{b}\sum_{\ell|a^2b}\frac{\mu(\ell)\chi_D(\ell)^2\psi\left(\frac{a^2b}{\ell};D\right)}{\ell}\\
&=\prod_{\substack{p\nmid qD\\p\nmid nq'}}\left(1-\frac{1}{p^2}\left(\psi(p^2;D)-\frac{\psi(p;D)}{p}\right)\right)\prod_{\substack{p|qD\\ p\nmid nq'}}\left(1-\frac{\chi_D(p)}{p}\left(\psi(p;D)-\frac{\chi_D(p)^2}{p}\right)\right)\\
&=\prod_{\substack{p\nmid qD\\p\nmid nq'}}\left(1-\frac{1}{p^2}\left(\psi(p^2;D)-\frac{\chi_D(p)\psi(p;D)}{p}\right)\right)\prod_{\substack{p|qD\\ p\nmid nq'}}\left(1-\frac{1}{p}\left(\psi(p;D)-\frac{\chi_D(p)}{p}\right)\right)\prod_{\substack{p|D\\p\nmid nq'}}\left(1-\frac{1}{p}\right)^{-1}
\end{flalign*}
We insert it into $\cM_2(q')$:
\begin{flalign*}
	\cM_2(q')&=X\widetilde{J}(1)L(1,\chi_D)\sum_{\substack{Dnq'=\square}}\frac{\lambda_f(n)}{\sqrt{n}}W\left(\frac{n}{Y}\right)\prod_{p|nq'}\left(1-\frac{1}{p}\right)\left(1-\frac{\chi_D(p)}{p}\right)\\ 
&\hspace{20mm}\prod_{\substack{p\nmid qD\\p\nmid nq'}}\left(1-\frac{1}{p^2}\left(\psi(p^2;D)-\frac{\chi_D(p)\psi(p;D)}{p}\right)\right)\prod_{\substack{p|qD\\ p\nmid nq'}}\left(1-\frac{1}{p}\left(\psi(p;D)-\frac{\chi_D(p)}{p}\right)\right)
\end{flalign*}
We then open the $W$-function by \eqref{eq. W function}:
\begin{flalign*}
	\cM_2(q')=\frac{X\widetilde{J}(1)L(1,\chi_D)\cG(q,D)}{2\pi i}\int_{(3)}\left(\frac{Y}{2\pi}\right)^w\Gamma(1+w)\cA_2(w;q',D) \frac{dw}{w^2}
\end{flalign*}
where $\cG(q,D)$ is defined in \eqref{eq. extra term}, and 
\begin{flalign*}
\cA_2(s;q',D):&=\sum_{\substack{Dnq'=\square}}\frac{\lambda_f(n)}{n^{1/2+s}}\prod_{\substack{p|nq'\\p\nmid qD}}\frac{(1-p^{-1})(1-\chi_D(p)p^{-1})}{1-p^{-2}(\psi(p^2;D)-\chi_D(p)\psi(p;D)p^{-1})}\prod_{\substack{p|nq'\\p|qD}}\frac{(1-p^{-1})(1-\chi_D(p)p^{-1})}{1-p^{-1}(\psi(p;D)-\chi_D(p)p^{-1})}\\
	&=\sum_{\substack{Dnq'=\square}}\frac{\lambda_f(n)}{n^{1/2+s}}\prod_{\substack{p|nq'\\p\nmid qD}}\frac{1-\chi_D(p)p^{-1}}{1+p^{-1}-(1+\chi_D(p))p^{-2}},
\end{flalign*}
which coincides with the series in Lemma \ref{lem. euler for negative q prime}. Then we shift the contour in $\cM_2(q')$ to $\Re(s)=-\frac{1}{8}$ and apply Lemma \ref{lem. euler for positive q prime}:
\[\cM_2(q')=\widetilde{J}(1)L(1,\chi_D)\cG(q,D)\cH_2(0,q',D)X\log X+O(X\log\log X),\]
and hence for $q'\in\{q,4q\}:$
\begin{equation}\label{eq. negative main}
	\cM(q')=\widetilde{J}(1)L(1,\chi_D)\cG(q,D)\cH_2(0,q',D)X\log X+O(X\log\log X).
\end{equation}
We insert \eqref{eq. positive main} and \eqref{eq. negative main} into \eqref{eq. main for moment}:
\[S(X)=\frac{(\cH_1(0,4,D)+\cH_1(0,4q,D)+\cH_2(0,-4,D)+\cH_2(0,-4q,D))}{2}\widetilde{J}(1)L(1,\chi_D)\cG(q,D)X\log X+O(X\log\log X).\]
By Remark \ref{rem. constant is positive}, the main term is positive, as expected.

\section{Proof of Proposition \ref{prop.difference prop}}\label{sec. proof of prop 1}
In this section, we prove Proposition \ref{prop.difference prop}. Let $G(x)$ be the smooth function defined in \cite[Equation (2.12)]{LiMR4768632}. By the dyadic division, we can write:
\[\sum_{d\in\cD}\psi(d,D)\cB(d;Y)J\left(\frac{d}{X}\right)=\ssum_{N}\sum_{d\in\cD}(1+\chi_d(q))\psi(d,D)J\left(\frac{d}{X}\right)\sum_{n\geq1}\frac{\lambda_f(n)\chi_d(n)}{\sqrt{n}}\left(W\left(\frac{n}{d\sqrt{q}}\right)-W\left(\frac{n}{Y}\right)\right)G\left(\frac{n}{N}\right).\] 
We consider the $N$-sum in three different ranges: $N>X,$ $Y< N\leq X$ and $N\leq Y,$ which yields:
\[\sum_{d\in\cD}\psi(d,D)\cB(d;Y)J\left(\frac{d}{X}\right)=\cS_{N>X}+\cS_{Y<N\leq X}+\cS_{N\leq Y}.\]
\subsection{The contribution of $N>X$: $\cS_{N>X}$}\label{subsec.con1}
When $N>X,$ we apply Cauchy-Schwarz inequality for the $d$-sum:
\[\cS_{N>X}\ll \ssum_{N>X}\left(\sum_{d\in\cD}\psi(d,D)^2J\left(\frac{d}{X}\right)\right)^{\frac{1}{2}}\left(\sum_{d\in\cD}J\left(\frac{d}{X}\right)\left|\sum_{n\geq1}G\left(\frac{n}{N}\right)\frac{\lambda_f(n)\chi_d(n)}{\sqrt{n}}\left(W\left(\frac{n}{d\sqrt{q}}\right)-W\left(\frac{n}{Y}\right)\right)\right|^2\right)^{\frac{1}{2}}.\]
Similar to the analysis of $S_2$-term in \cite[Section 7.5]{zhou2025momentderivativesquadratictwists}, we can show:
\[\sum_{d\in\cD}J\left(\frac{d}{X}\right)\left|\sum_{n\geq1}G\left(\frac{n}{N}\right)\frac{\lambda_f(n)\chi_d(n)}{\sqrt{n}}\left(W\left(\frac{n}{d\sqrt{q}}\right)-W\left(\frac{n}{Y}\right)\right)\right|^2\ll \left(\frac{X}{N}\right)^6X.\]
On the other hand, by the properties of $\psi(m,D)$, we apply the Mellin inverse transformation for the  compactly supported function $J(\cdot)$:
\[\sum_{d\in\cD}\psi(d,D)^2J\left(\frac{d}{X}\right)\leq\sum_{d\geq1}\psi(d,D)^2J\left(\frac{d}{X}\right)=\frac{1}{2\pi i}\int_{(2)}\widetilde{J}(s)X^{s}L_D(s)\,ds,\]
where
\[L_D(s)=\sum_{d\geq1}\frac{\psi(d,D)^2}{d^s}=\zeta(s)^2L(s,\chi_D)^2\mathcal{H}(s;D),\]
where $\mathcal{H}(s;D)$ is an Euler product absolutely convergent for $\Re(s) > \frac{1}{2}$.
We shift the contour to $\Re(s)=\frac{3}{4}.$ Noticing that there is a second order pole at $s=1$ for $L_D(s),$ we obtain:
\begin{equation}\label{eq. sec mom of psi}	
	\sum_{d\in\cD}\psi(d,D)^2J\left(\frac{d}{X}\right)\ll X\log X.
\end{equation}
This will imply, together with $\ssum_{N>X}\frac{X^3}{N^3}\ll1,$
\[S_{N>X}\ll (X\log X)^{\frac{1}{2}}X^{\frac{1}{2}}=X\log^{\frac{1}{2}} X.\]

\subsection{The contribution of $Y<N\leq X$: $\cS_{Y<N\leq X}$}\label{subsec.con2}
When $Y<N\leq X,$ we apply Cauchy-Schwarz inequality for the $d$-sum:
\[\cS_{Y<N\leq X}\ll \ssum_{Y<N\leq X}\left(\sum_{d\in\cD}\psi(d,D)^2J\left(\frac{d}{X}\right)\right)^{\frac{1}{2}}\left(\sum_{d\in\cD}J\left(\frac{d}{X}\right)\left|\sum_{n\geq1}G\left(\frac{n}{N}\right)\frac{\lambda_f(n)\chi_d(n)}{\sqrt{n}}\left(W\left(\frac{n}{d\sqrt{q}}\right)-W\left(\frac{n}{Y}\right)\right)\right|^2\right)^{\frac{1}{2}}.\]
Similar to the analysis of $S_1$-term in \cite[Lemma 7.1]{zhou2025momentderivativesquadratictwists}, we can show:
\[\sum_{d\in\cD}J\left(\frac{d}{X}\right)\left|\sum_{n\geq1}G\left(\frac{n}{N}\right)\frac{\lambda_f(n)\chi_d(n)}{\sqrt{n}}\left(W\left(\frac{n}{d\sqrt{q}}\right)-W\left(\frac{n}{Y}\right)\right)\right|^2\ll \left(\frac{X}{N}\right)^{2/\log X}X(\log\log X)^4.\]
This will imply, together with \eqref{eq. sec mom of psi} and the fact that $\ssum_{Y<N\leq X}\left(\frac{X}{N}\right)^{1/\log X}\ll \log\log X$,
\[S_{Y<N\leq X}\ll X(\log X)^{\frac{1}{2}}(\log\log X)^3.\]

\subsection{The contribution of $N\leq X$: $\cS_{N\leq Y}$}\label{subsec.con3}
The investigation of $\cS_{N\leq Y}$ is similar to that of Proposition \ref{prop.truncated sum} but more involved. 

We choose $Z=\log^{20} X.$ Similar to the study of Proposition \ref{prop.truncated sum}, we insert the M\"obius function to detect the squarefree condition. We will also have the $a\leq Z$ contribution and $a>Z$ contribution. For the $a>Z$ part, we again apply Cauchy-Schwarz inequality and we can bound this part by $O\left(\frac{X\log^4 X}{Z^{1-\varepsilon}}\right).$ For the $a\leq Z$ term, we insert the M\"obius functions to drop the copime conditions on $d,$ and apply Lemma \ref{lem. Hecke for psi}, which implies: 
\begin{equation}\label{eq. small range}
\cS_{N\leq Y}=\frac{1}{2}\sum_{q'\in\{\pm4,\pm 4q\}}\cS_{q'}(a\leq Z)+O\left(\frac{X\log^4 X}{Z^{1-\varepsilon}}\right),	
\end{equation}
where
\begin{equation}
	\begin{split}
		\cS_{q'}(a\leq Z)& = \sum_{\substack{a \leq Z \\ (a,2qD)=1}}\mu(a)\sum_{b|qD}\mu(b)\sum_{\ell|a^2b}\mu(\ell)\chi_D(\ell)\psi\left(\frac{a^2b}{\ell};D\right)\\
	&\hspace{2mm} \times \ssum_{N\leq Y}\sum_{(n,a)=1}\frac{\lambda_f(n)\chi_{nq'}(b\ell)}{\sqrt{n}}G\left(\frac{n}{N}\right)\sum_{d\geq 1}\psi(d;D)\chi_{nq'}(d)J\left(\frac{a^2bd\ell}{X}\right)\left(W\left(\frac{n}{a^2bd\ell\sqrt{q}}\right)-W\left(\frac{n}{Y}\right)\right).
	\end{split}
\end{equation}
We set:
\[\cL(q'):=\ssum_{N\leq Y}\sum_{(n,a)=1}\frac{\lambda_f(n)\chi_{nq'}(b\ell)}{\sqrt{n}}G\left(\frac{n}{N}\right)\sum_{d\geq 1}\psi(d;D)\chi_{nq'}(d)J\left(\frac{a^2bd\ell}{X}\right)\left(W\left(\frac{n}{a^2bd\ell\sqrt{q}}\right)-W\left(\frac{n}{Y}\right)\right).\]
Let $\varepsilon=\frac{\log\log X}{\log X}$ and $C_{\varepsilon}$ be the contour:
\[C_{\varepsilon}=\{it:|t|\geq\varepsilon\}\bigcup\{\varepsilon  e^{i\theta}:\theta\in[-\pi/2,\pi/2]\}.\]
We open the $W$-function by definition and then apply the Mellin inverse transformation for $J(x)$:
\begin{flalign*}
\cL(q')=\frac{1}{(2\pi i)^2}\ssum_{N\leq Y}\int_{C_{\varepsilon}}\int_{(3)}&\frac{\Gamma(1+w)}{(2\pi)^w}\frac{\widetilde{J}(s)X^s}{(a^2b\ell)^s}\sum_{(n,a)=1}\frac{\lambda_f(n)\chi_{nq'}(b\ell)}{n^{1/2+w}}G\left(\frac{n}{N}\right)\\
&\left((a^2b\ell\sqrt{q})^{w}L(s-w,\chi_{Dnq'})L(s-w,\chi_{nq'})-Y^wL(s,\chi_{Dnq'})L(s,\chi_{nq'})\right)\,ds\frac{\,dw}{w^2}.	
\end{flalign*}
We apply the change of variable $s\mapsto s+w$ for $L(s-w,\chi_{Dnq'})L(s-w,\chi_{nq'})$:
\begin{flalign*}
\cL(q')=\frac{1}{(2\pi i)^2}\ssum_{N\leq Y}\int_{C_{\varepsilon}}&\int_{(3)}\frac{\Gamma(1+w)}{(2\pi)^w}\sum_{(n,a)=1}\frac{\lambda_f(n)\chi_{nq'}(b\ell)}{n^{1/2+w}}G\left(\frac{n}{N}\right)\\
&\left(\frac{\widetilde{J}(s+w)X^{s+w}q^{w/2}}{(a^2b\ell)^{s}}L(s,\chi_{Dnq'})L(s,\chi_{nq'})-\frac{\widetilde{J}(s)X^sY^w}{(a^2b\ell)^s}L(s,\chi_{Dnq'})L(s,\chi_{nq'})\right)\,ds\frac{\,dw}{w^2}.	
\end{flalign*}
We shift the contour of $\Re(s)=3$ to $\Re(s)=\frac{1}{2}$, which implies:
\[\cL(q')=\cL_1(q')+\cL_2(q'),\] 
where
\begin{flalign*}
\cL_1(q')=\frac{1}{(2\pi i)^2}\ssum_{N\leq Y}\int_{C_{\varepsilon}}&\int_{(1/2)}\frac{\Gamma(1+w)}{(2\pi)^w}\sum_{(n,a)=1}\frac{\lambda_f(n)\chi_{nq'}(b\ell)}{n^{1/2+w}}G\left(\frac{n}{N}\right)\\
&\left(\frac{\widetilde{J}(s+w)X^{s+w}q^{w/2}}{(a^2b\ell)^{s}}L(s,\chi_{Dnq'})L(s,\chi_{nq'})-\frac{\widetilde{J}(s)X^sY^w}{(a^2b\ell)^s}L(s,\chi_{Dnq'})L(s,\chi_{nq'})\right)\,ds\frac{\,dw}{w^2}.	
\end{flalign*}
and
\begin{flalign*}
\cL_2(q')&=\frac{X}{a^2b\ell}\ssum_{N\leq Y}\frac{1}{2\pi i}\int_{C_{\varepsilon}}\frac{\Gamma(1+w)}{(2\pi)^w}(\widetilde{J}(1+w)X^wq^{w/2}-\widetilde{J}(1)Y^w)\\
&\hspace{35mm}\times\sum_{(n,a)=1}\frac{\lambda_f(n)\chi_{nq'}(b\ell)}{n^{1/2+w}}G\left(\frac{n}{N}\right)\Res_{s=1}L(s,\chi_{Dnq'})L(s,\chi_{nq'})\frac{\,dw}{w^2}.	
\end{flalign*}

For $\cL_1(q'),$ we change the order of integration. By the choice of $\varepsilon=\frac{\log\log X}{\log X},$ we obtain:
\[\cL_1(q')\ll \frac{X^{\frac{1}{2}}\log^3 X}{(a^2b\ell)^{\frac{1}{2}}}\ssum_{N\leq Y}\int_{-\infty}^{\infty}\sum_{n\ll N^{1+\varepsilon}}\frac{|\lambda_f(n)|}{n^{1/2}}G\left(\frac{n}{N}\right)|L(1/2+it,\chi_{Dnq'})||L(1/2+it,\chi_{nq'})|\,dt\]
Similar to \S \ref{subsec. error} (we notice that $G(n/N)\ll H(N)$, where $H(N)$ is defined in \eqref{eq. growth of W}), we can show that
\[\cL_1(q')\ll \frac{(XY)^{\frac{1}{2}}\log^9Y}{(a^2b\ell)^{\frac{1}{2}}}.\] 
We insert $\cL_1(q')$ and $\cL_2(q)'$ into $\cL(q')$ and then into $\cS_{q'}(a\leq Z)$:
\[\cS_{q'}(a\leq Z)= \sum_{\substack{a \leq Z \\ (a,2qD)=1}}\mu(a)\sum_{b|qD}\mu(b)\sum_{\ell|a^2b}\mu(\ell)\chi_D(\ell)\psi\left(\frac{a^2b}{\ell};D\right)\cL_2(q')+O\left(\frac{X}{\log^{10}X}\right).\]
This will imply, by \eqref{eq. small range},
\[\cS_{N\leq Y}=\frac{1}{2}\sum_{q'\in\{\pm4,\pm 4q\}}\sum_{\substack{a \leq Z \\ (a,2qD)=1}}\mu(a)\sum_{b|qD}\mu(b)\sum_{\ell|a^2b}\mu(\ell)\chi_D(\ell)\psi\left(\frac{a^2b}{\ell};D\right)\cL_2(q')+O\left(\frac{X}{\log^{10}X}\right)+O\left(\frac{X\log^4 X}{Z^{1-\varepsilon}}\right).\]
We apply the Mellin inverse transformation for $G$-function, and then follow the argument in \S \ref{subsec. main term}, which yields:

\small
\begin{flalign*}
	\cS_{N\leq Y}=&\frac{XL(1,\chi_D)\cG(q,D)}{2(2\pi i)^2}\sum_{q'\in\{4,4q\}}\ssum_{N\leq Y}\int\limits_{(3)}\int\limits_{C_{\varepsilon}}\widetilde{G}(u)N^u\left(\frac{\Gamma(1+w)}{(2\pi)^w}\right)(\widetilde{J}(1+w)X^wq^{w/2}-\widetilde{J}(1)Y^w)\cA_1(w+u;q',D) \frac{dw}{w^2}\,du\\
	&+\frac{XL(1,\chi_D)\cG(q,D)}{2(2\pi i)^2}\sum_{q'\in\{-4,-4q\}}\ssum_{N\leq Y}\int\limits_{(3)}\int\limits_{C_{\varepsilon}}\widetilde{G}(u)N^u\left(\frac{\Gamma(1+w)}{(2\pi)^w}\right)(\widetilde{J}(1+w)X^wq^{w/2}-\widetilde{J}(1)Y^w)\cA_2(w+u;q',D) \frac{dw}{w^2}\,du\\
	&+O\left(\frac{X}{\log^{10}X}\right)+O\left(\frac{X\log^4 X}{Z^{1-\varepsilon}}\right).
\end{flalign*}

\normalsize
We first move the integration line of $u$ to $\Re(u)=-\frac{1}{16}$, and then move the integration line of $w$ to $\Re(w)=-\frac{1}{16}$. By Lemma \ref{lem. euler for positive q prime} and Lemma \ref{lem. euler for negative q prime}, the integrand will only have a simple pole at $w=0.$ Moreover, we can show that
\[\Res_{w=0}\cdots=\widetilde{J}(1)(\log X-\log Y)+O(1)\ll \log\log X,\]
since $Y=\frac{X}{\log^{100} X}.$ This will imply:
\[\cS_{N\leq Y}\ll X\log\log X\ssum_{N\leq Y}N^{-1/16}+X^{15/16}\ssum_{N\leq Y}N^{-1/16}+O\left(\frac{X}{\log^{10}X}\right)+O\left(\frac{X\log^4 X}{Z^{1-\varepsilon}}\right)\ll X\log\log X.\]
We combine the results in \S \ref{subsec.con1}, \S \ref{subsec.con2} and \S \ref{subsec.con3}:
\[\sum_{d\in\cD}\psi(d,D)\cB(d;Y)J\left(\frac{d}{X}\right)\ll X(\log X)^{\frac{1}{2}}(\log\log X)^3.\]
This completes the proof of Proposition \ref{prop.difference prop}.

\section*{Acknowledgment}
We would like to express our gratitude to Yongxiao Lin, Sheng-Chi Liu, Tong Wei and Shuai Zhai for their encouragements and valuable suggestions.

\bibliographystyle{alpha}	
\bibliography{TL}

\end{document}